\documentclass[reqno, 11pt]{amsart}
\usepackage[colorlinks]{hyperref}
\usepackage{url}
\usepackage{times}

\newtheorem{theorem}{Theorem}
\newtheorem{lemma}[theorem]{Lemma}

\newtheorem{corollary}[theorem]{Corollary}

\theoremstyle{definition}

\newtheorem{remark}[theorem]{Remark}

\usepackage{bbm}
\usepackage{url}
\usepackage{euscript}
\usepackage{pb-diagram}
\usepackage{lamsarrow}
\usepackage{pb-lams}
\usepackage{amsmath}
\usepackage{amsthm}

\usepackage{epsfig}
\usepackage{amssymb}
\usepackage{pifont}
\usepackage{amsbsy}

\usepackage{graphicx}
\usepackage{epstopdf}

\newskip\aline \newskip\halfaline
\def\skipaline{\vskip\aline}

\def\qedbox{$\rlap{$\sqcap$}\sqcup$}
\def\qed{\nobreak\hfill\penalty250 \hbox{}\nobreak\hfill\qedbox\skipaline}

\newcommand\bR{{\mathbb R}}

\newcommand{\be}{\boldsymbol{e}}

\newcommand{\ve}{{\varepsilon}}

\newcommand{\eB}{\EuScript{B}}

\newcommand{\eM}{\EuScript{M}}

\newcommand{\eT}{\EuScript{T}}

\newcommand{\ra}{\rightarrow}

\newcommand{\Llra}{{\Longleftrightarrow}}

\def\inpr{\mathbin{\hbox to 6pt{\vrule height0.4pt width5pt depth0pt \kern-.4pt \vrule height6pt width0.4pt depth0pt\hss}}}

\begin{document}

\title{Bead sliding and convex inequalities}

\author{Liviu I. Nicolaescu}
\thanks{Last modified on {\today}.}

\address{Department of Mathematics, University of Notre Dame, Notre Dame, IN 46556-4618.}
\email{nicolaescu.1@nd.edu}
\urladdr{\url{http://www.nd.edu/~lnicolae/}}

\begin{abstract}  We analyze  a  simple  game of beads on a rod and relate it to some classical convex inequalities.. \end{abstract}

\maketitle

We consider distributions (or configurations) of $n$ beads  on the  real semiaxis  $[\mu,\infty)$.   Any bead in such a distribution is  capable of sliding to the right (in the positive  direction) but not allowed to slide to the left. We indicate such a  distribution of beads by a vector
\[
\vec{A}=(A_1,\dotsc, A_n),\;\; \mu \leq A_1<A_2<\cdots <A_n,
\]
where the coordinates   $A_i$ indicate the positions of the beads.    The $i$-th bead is the bead located at $A_i$. A distribution is called \emph{monotone} if
\[
A_1-\mu \leq A_2-A_1\leq \cdots \leq A_n-A_{n-1}.
\]
We denote by $\eB_n=\eB_n(\mu)$ the collection of    monotone  distributions of  $n$ beads on the semiaxis $[\mu,\infty)$.      Clearly, we can view  $\eB_n(\mu)$ as a closed convex set in $\bR^n$.   

We will indicate  the elements  of $\eB_n(\mu)$ using capital letters $\vec{A}$, $\vec{B}$ etc. To a configuration $\vec{A}\in \eB_n(\mu)$ we associate the vector of differences $\vec{a}=\Delta\vec{A}$,
\[
\vec{a}=(a_1,\dotsc, a_n), \;\;a_1=A_1-\mu,\dotsc, a_k=A_k-A_{k-1},\forall k=2,\dotsc, n.
\]
We have  a natural  partial order on $\eB_n(\mu)$
\[
\vec{A} \leq \vec{B} \Llra A_k\leq B_k,\;\;\forall k=1,\dotsc, n.
\]
 Let $\be_1,\dotsc, \be_n$ denote the canonical basis of $\bR^n$. Given a bead distribution  $\vec{A}\in \eB_n(\mu)$ we define an \emph{admissible  bead  slide} to be a transformation
 \[
 \vec{A}\mapsto  \vec{A'}=\vec{A}+\delta\be_k,
 \]
 where $\delta\geq 0$, $1\leq k\leq n$ and  the distribution $\vec{A'}$ is  monotone.  Intuitively, this means that  we slide to the right by a distance $\delta$   the  $k$-th bead  of the  distribution  $\vec{A}$. The  admissibility of the move   means that the resulting distribution of  beads continues to be monotone.

 We define a new partial relation  $\preceq$ on $\eB_n(\mu)$  by declaring $\vec{A}\preceq \vec{B}$ if  the distribution $\vec{B}$ can be obtained  from $\vec{A}$   via a finite sequence of  admissible bead slides.   When  $\vec{A}\preceq \vec{B}$ we say that \emph{we can slide the distribution} $\vec{A}$ to the distribution $\vec{B}$

 If we think   of $\eB_n(\mu)$ as a closed convex set in $\bR^n$ and $\vec{A},\vec{B}\in\eB_n(\mu)$, then $\vec{A}\preceq \vec{B}$ if and only if   we can travel  from $\vec{A}$ to $\vec{B}$ \emph{inside}  $\eB_n(\mu)$ along   a positive zig-zag,  i.e.,  a continuous path   consisting of finitely many   segments     parallel to  the   coordinate  axes and oriented in the positive directions of the axes. 
 
 The goal of this note  is to investigate when  can we slide one mononote  distribution of beads to another monotone distribution.  Clearly if we can slide $\vec{A}$ to $\vec{B}$ then $\vec{A}\leq \vec{B}$.
 
\begin{remark} The converse implication  is true if $n=1,2$, but false if $n\geq 3$. Indeed if $n\geq 3$, and $\vec{B}\in\eB_n(\mu)$ is an equidistant distribution, i.e.,
 \[
 B_1-\mu=B_2-B_1=\cdots=B_n-B_{n-1}
 \]
 then there is no distribution $\vec{A}\prec\vec{B}$.  To see this observe that there is no distribution $\vec{A}$ such that $\vec{B}$ is obtained from $\vec{A}$ by a single admissible bead slide.   \qed
 \label{rem: equi}
 \end{remark}

 Define $\lambda_n: \eB_{n}(\mu) \ra [0,\infty)$ by setting
 \[
 \lambda_n(\vec{A}):=   a_{n}-a_1=  (a_{k}- a_{n-1})+\cdots (a_2-a_1) +a_1,
 \]
 where we recall that
 \[
a_1=A_1-\mu,\;\; a_k= A_k-A_{k-1},\;\;k>1.
 \]
 Clearly $\lambda_n(\vec{A})=0$ if and only      the  beads  described by the distribution $\vec{A}$ are equidistant, i.e.,
 \[
 A_n-A_{n-1}=\cdots =A_2-A_1=A_1-\mu.
 \]
 The following is the main result of this note.

\begin{theorem} Let $\mu\in \bR$ and $\vec{B}\in\eB_n(\mu)$. Then
 \begin{equation}
  b_k>b_{k-2},\;\;\forall k\geq 3\Longleftrightarrow \vec{A}\preceq \vec{B},\;\;\forall \vec{A}<\vec{B}\in\eB_n(\mu),
\label{eq: order}
 \end{equation}
 where
 \[
 b_1=B_1-\mu,\;\;b_k=B_k-B_{k-1},\;\;\forall k\geq 2.
 \]
 \label{th: order}
 \end{theorem}

\begin{remark} The condition $b_k> b_{k-2}$, $\forall k\geq 3$  signifies  that no string of  four  consecutive beads of the the distribution $\vec{B}$  is equidistant.\qed
\end{remark}

 \begin{proof}  We first   prove the  implication $\Rightarrow$,
 \begin{equation*}
  b_k>b_{k-2},\;\; \forall k\geq 3\Longrightarrow \vec{A}\preceq \vec{B},\;\;\forall \vec{A}\leq \vec{B}.
 \tag{$S_n$}
 \label{tag: sn}
 \end{equation*}
 We argue  by induction on $n$. The cases $n=1$  and $n=2$ are trivial.
 
 To  complete the inductive step note  first that the assumption $b_k>b_{k-2}$ forall $k\geq 2$ implies $\lambda(\vec{B})>0$.  We have the following key estimate.

 \begin{lemma}   If $\vec{A},\vec{B}\in \eB_{n+1}(\mu)$ are such that $\vec{A}\leq \vec{B}$ and $A_{n+1}=B_{n+1}$ then
 \begin{equation}
 \lambda_{n+1}(\vec{A})\geq \frac{1}{n}\lambda_{n+1}(\vec{B}).
 \label{eq: ineq}
 \end{equation}
 \end{lemma}

 \begin{proof} For $k=2,\dotsc, n+1$ we set
 \[
 \alpha_k:= a_k-a_{k-1},\;\;\beta_k:=b_k-b_{k-1}.
 \]
 Note that  $\alpha_k,\beta_k\geq 0$,
 \[
 \lambda_{n+1}(\vec{A)}=\sum_{k=2}^{n+1}\alpha_k,\;\;\lambda_{n+1}(\vec{B})=\sum_{k=2}^{n+1} \beta_k,
 \]
 \[
 a_k=  a_1+\sum_{i=2}^k \alpha_i,\;\; b_k =a_1+\sum_{i=2}^k\beta_i,
 \]
 and
 \[
 (n+1)a_1+\sum_{k=2}^{n+1}(n-k+1) \alpha_k=A_{n+1}-\mu=B_{n+1}-\mu=(n+1)b_1+\sum_{k=2}^{n+1}(n-k+1) \beta_k.
 \]
 Hence
 \[
 \sum_{k=2}^{n+1}(n-k+1) \alpha_k=(n+1)(b_1-a_1)+\sum_{k=2}^{n+1}(n-k+1) \beta_k\geq \sum_{k=2}^{n+1}(n-k+1) \beta_k.
 \]
We deduce
\[
n\lambda_{n+1}(\vec{A})=n\sum_{k=2}^{n+1}\alpha_k\geq \sum_{k=2}^{n+1}(n-k+1) \alpha_k\geq \sum_{k=2}^{n+1}(n-k+1) \beta_k\geq \sum_{k=2}^{n+2}\beta_k=\lambda_{n+1}(\vec{B}).
\]
 \end{proof}

 Consider two  distributions $\vec{A},\vec{B}\in \eB_{n+1}(\mu)$.        Then  we can slide the last bead of $\vec{A}$ until it reaches the  position of the last bead of $\vec{B}$.
 \[
 \vec{A}\mapsto \vec{A'}:=\vec{A}+ \bigl(\, B_{n+1}-A_{n+1}\,\bigr)\be_{n+1}
 \]
 Clearly this    slide is admissible.      This shows that it suffices to prove ($S_{n+1}$) only in the special case  $A_{n+1}=B_{n+1}$.        To prove the implication ($S_{n+1}$) we will rely on the following simple observation.

 \begin{lemma} Assume that  the implication  ($S_k$) holds for every $k\leq n$.  If  $\vec{A},\vec{B}\in \eB_{n+1}(\mu)$ are two distributions    such that  $\vec{A}\leq \vec{B}$,  and $A_k=B_k$  for some $k\leq n$ then $\vec{A}\preceq \vec{B}$.
\label{lemma: ind}
 \end{lemma}

 \begin{proof}      Note that
 \[
(A_1,\dotsc, A_k)\leq (B_1,\dotsc, B_k)\;\;\mbox{and}\;\; (A_{k+1},\dotsc, A_{n+1})\leq (B_{k+1},\dotsc, B_{n+1}).
\]
According to $S_k$,  we can slide  the first $k$-beads of the distribution $\vec{A}$ to the first $k$ beads of the distribution  $\vec{B}$.  Using $S_{n-k+1}$ we can then  slide the last $(n-k+1)$ beads of the distribution $\vec{A}$ to  the last $(n-k+1)$ beads of the distribution $\vec{B}$.

 \end{proof}

 Using the   above   observations we  deduce  that  the implication  $S_{n+1}$ is a consequence of the   following result.

 \begin{lemma} Assume that  the implication  ($S_k$) holds for every $k\leq n$.  If  $\vec{A},\vec{B}\in \eB_{n+1}(\mu)$ are two distributions    such that $\vec{A}\leq \vec{B}$ and $A_{n+1}=B_{n+1}$ then we can slide $\vec{A}$ to a   configuration $\vec{C}\in \eB_{n+1}(\mu)$  that \emph{crosses}  $\vec{B}$, i.e.,

 \begin{enumerate}

 \item[(a)] $\vec{C}\leq \vec{B}$,

 \item[(b)] $C_{n+1}=B_{n+1}$,

 \item[(c)] $C_k=B_k$  for some $k\leq n$.

 \end{enumerate}

 \label{lemma: cross}
 \end{lemma}

 \begin{proof}   Define
 \[
 \eB_{n+1}(\vec{B}):=\bigl\{\, \vec{T}\in \eB_{n+1}(\mu);\;\;\vec{T}\leq \vec{B},\;\;T_{n+1}=B_{n+1}\,\bigr\}.
 \]
Note that $\vec{A}\in \eB_{n+1}(\vec{B})$.  We  define a $\vec{B}$-move, to be a bead slide on a configuration $\vec{T}\in \eB_{n+1}(\vec{B})$ that produces another configuration in $\eB_{n+1}(\vec{B})$.   We need to prove  that by   a  sequence  of $\vec{B}$-moves starting with $\vec{A}$ we can produce  a    configuration  $\vec{C}\in \eB_{n+1}(\vec{B})$ that crosses $\vec{B}$.

We argue  by contradiction. Hence we   will work under the following assumption.
\begin{equation*}
\mbox{ \emph{We cannot produce crossing configurations  via any sequence of  $\vec{B}$-moves starting with $\vec{A}$.}}
\tag{$\dag$}
\label{tag: a}
\end{equation*}
 We  show that this implies that there exists  a   sequence of configurations $\vec{A}_\nu\in \eB_{n+1}(\vec{B})$, $\nu\geq 1$, such that
\[
\lim_{\nu\ra \infty}\lambda_{n+1}(\vec{A}_\nu)=0.
\]
In view of the assumption $\lambda(\vec{B})>0$ this sequence  contradicts the inequality (\ref{eq: ineq}).

Denote by $\eB_{n+1}(\vec{A},\vec{B})$ the set of configurations in $\eB_{n+1}(\vec{B})$ that can be obtained from $\vec{A}$ by a sequence of $\vec{B}$-moves. We  will  produce a real number $\kappa\in (0,1)$ and a map
\[
\eT:\eB_{n+1}(\vec{A},\vec{B})\ra \eB_{n+1}(\vec{A},\vec{B})
\]
such that
\[
\lambda\bigl(\, \eT(\vec{X})\,\bigr)\leq \kappa \lambda(\vec{X}),\;\;\forall \vec{X}\in\eB_{n+1}(\vec{A},\vec{B}).
\]
The sequence
\[
\vec{A}_\nu:=\eT^\nu(\vec{A})
\]
will then produce the sought for contradiction.

We begin by constructing maps
\[
\eM_1,\eM_2,\dotsc, \eM_n:\eB_{n+1}(\mu)\ra \eB_{n+1}(\mu)
\]
so that for any $k=1,\dotsc, n$ and any $\vec{X}\in\eB_{n+1}(\mu)$ we have
\[
\eM_k(\vec{X})=\Bigl(\, X_1,\dotsc, x_{k-1}, \frac{1}{2}(X_{k-1}+X_{k+1}), X_{k+1},\dotsc, X_{n+1}\,\Bigr),
\]
where for uniformity we set $X_0=\mu$. In other words, $\eM_k(\vec{X})$ is obtained from $\vec{X}$ by sliding the $k$-th bead of $\vec{X}$ to the midpoint of the interval $(X_{k-1}, X_k)$. In the new configuration the beads $(k-1)$, $k$ and  $(k+1)$ are equidistant.

Now define
\[
\eT:\eB_{n+1}(\mu)\ra \eB_{n+1}(\mu),\;\;\eT=\eM_1\circ \eM_2\circ\cdots \circ \eM_n.
\]
Note that
\[
\eM_n(X_1,\dotsc, X_{n+1})=\Bigl(X_1,\dotsc, X_{n-1}, \frac{1}{2}(X_{n-1}+X_{n+1}), X_{n+1}\,\Bigr).
\]
The configuration $\eM_{n-1}\circ \eM_n(\vec{X})$ differs from $\eM_n(X)$ only at the $(n-1)$-th component which is
\[
 \frac{1}{2}X_{n-2}+\frac{1}{4}X_{n-1}+\frac{1}{4}X_{n+1}.
\]
The $(n-k)$-th component of $\eM_{n-k}\circ\cdots \circ \eM_n(\vec{X})$ is
\[
\frac{1}{2}X_{n-k-1}+\frac{1}{4}X_{n-k}+\cdots +\frac{1}{2^{k+1}}X_{n-1}+\frac{1}{2^{k+1}}X_{n+1}.
\]
The first component of $\vec{Y}:=\eT(\vec{X})$ is
\[
Y_1=\frac{1}{2}X_0+\frac{1}{4}X_1+\cdots +\frac{1}{2^n}X_{n-1}+\frac{1}{2^n}X_{n+1}.
\]
If we set
\[
x_1=X_1-X_0=X_1-\mu, \;\;x_2=X_2-X_1,\dotsc, x_{n+1}=X_{n+1}-X_n
\]
we deduce
\[
Y_1 =\frac{1}{2^n}X_{n+1}+\sum_{k=0}^{n-1}\frac{1}{2^{k+1}}X_k=\frac{1}{2^n}X_{n+1}+\sum_{k=0}^{n-1}\frac{1}{2^{k+1}}\Bigl(\mu+\sum_{i=1}^k x_i\Bigr)
\]
\[
=\frac{1}{2^{n}}\Bigl(\mu+\sum_{i=1}^{n+1}x_i\Bigr) +(1-\frac{1}{2^n})\mu+ \sum_{k=1}^{n-1}\frac{1}{2^{k+1}}\sum_{i=1}^k x_i
\]
\[
=\mu+\frac{1}{2^n}\sum_{i=1}^{n+1}x_i+ \Bigl(\sum_{k=1}^{n-1}\frac{1}{2^{k+1}}\Bigr)x_1+\Bigl(\sum_{k=2}^{n-1}\frac{1}{2^{k+1}}\Bigr)x_2+\cdots +\frac{1}{2^n}x_{n-1}
\]
\[
= \mu+\frac{1}{2}x_1+\frac{1}{4}x_2+\cdots +\frac{1}{2^{n-1}}x_{n-1}+\frac{1}{2^n}x_n+\frac{1}{2^n}x_{n+1}.
\]

Observe that
\[
\lambda_{n+1}(X)= x_{n+1}-x_1,\;\;\lambda_{n+1}(Y)=y_{n+1}-y_1=Y_{n+1}-Y_n-Y_1+Y_0.
\]
We have
\[
\lambda_{n+1}(Y)= X_{n+1}- \frac{1}{2}(X_{n+1}+X_{n-1})-Y_1+\mu
\]
\[
=\sum_{i=1}^{n+1}x_i-\frac{1}{2}\Bigl(\sum_{i=1}^{n+1}x_i+\sum_{i=1}^{n-1}x_i\Bigr)-\Bigl(\frac{1}{2^n}x_{n+1}+\sum_{k=1}^n\frac{1}{2^k}x_k\Bigr)
\]
\[
=\frac{1}{2}(x_{n+1}+x_n)-\Bigl(\frac{1}{2^n}x_{n+1}+\sum_{k=1}^n\frac{1}{2^k}x_k\Bigr)
\]
\[
\leq \bigl(\,1-\frac{1}{2^n}\,\bigr)x_{n+1}-\sum_{k=1}^n\frac{1}{2^k}x_k=\sum_{k=1}^n\frac{1}{2^k}(x_{n+1}-x_k)
\]
\[
\leq \Bigl(\,\sum_{k=1}^n \frac{1}{2^k}\,\Bigr)(x_{n+1}-x_1)=\bigl(\,1-\frac{1}{2^n}\,\bigr)\lambda_{n+1}(X).
\]
Hence
\begin{equation}
\lambda_{n+1}\bigl(\,\eT(\vec{X})\,\bigr)\leq (1-2^{-n})\lambda_{n+1}(\vec{X}),\;\;\forall \vec{X}\in\eB_{n+1}(\mu).
\label{eq: contr}
\end{equation}
To conclude the proof  it suffices to show that
\begin{equation}
\eM_k(\vec{X})\in\eB(\vec{B}),\;\;\forall \vec{X}\in \eB(\vec{A},\vec{B}),\;\;k=1,\dotsc, n.
\label{eq: non-cross}
\end{equation}
Let $\vec{X}=(X_1,\dotsc, X_{n+1})\in \eB(\vec{A},\vec{B})$ and set  $\vec{Y}=\eM_k(\vec{X})$. Then
\[
Y_i= \begin{cases}
X_i, &i\neq k\\
\frac{1}{2}(X_{k-1}+X_{k+1}), &i=k.
\end{cases}
\]
To prove that $\vec{Y}\in \eB(\vec{A},\vec{B})$ we have to prove that  $Y_k\leq B_k$.    If this were not the case, then $Y_k>B_k$. Since $X_k<B_k$, we deduce $(B_k-X_k)<(Y_k-X_k)$.  This implies that  sliding the
 the $k$-th bead of $\vec{X}$ by  distance  $(B_k-X_k)$ is an admissible slide, and it is obviously a $\vec{B}$-move since the resulting configuration $\vec{X'}$  is in $\eB_{n+1}(\vec{B})$.  Clearly, the configuration $\vec{X'}$  crosses $\vec{B}$ since $X'_k=B_k$.  This contradicts the assumption (\ref{tag: a}) and finishes the proof of Lemma \ref{lemma: cross}  and of the implication $\Rightarrow$ in (\ref{eq: order}).

 \end{proof}
 
  To prove the converse implication $\Leftarrow$  we argue by induction.   The cases $n=1,2$ are trivial, while the case $n=3$ follows from Remark \ref{rem: equi}. 
  
  For the inductive step  suppose  $\vec{A}\prec \vec{B}$ in $\eB_{n+1}(\mu)$, $\forall \vec{A}<\vec{B}$.  Then
  \[
  (B_1,\dotsc, A_n)\prec(B_1,\dotsc, B_n)\in\eB_n(\mu),\;\;\forall (A_1,\dotsc, A_n)<(B_1,\dotsc, B_n),
  \]
  and the inductive assumption implies
  \[
  b_k>b_{k-2},\;\;\forall 2\leq k\leq n.
  \]
  To prove that $b_{n+1}>b_{n-1}$ we  argue by contradiction. Suppose $b_{n+1}=b_{n-1}$ so that
  \[
  b_{n+1}=b_n=b_{n-1}. 
  \]
The  condition $b_n>b_{n-2}$ implies that  $b_{n-2}<b_{n-1}$. Consider the bead distribution $\vec{C}\in\eB_{n+1}(\mu)$ described by
\[
C_k=B_k,\;\;\forall k\leq n-2, 
\]
\[
C_{n-1}=C_{n-2}+b_{n-2}=B_{n-2}+b_{n-2}< B_{n-1},
\]
\[
C_n=C_{n-1}+b_{n-1}<B_n,\;\;C_{n+1}=C_{n}+b_n<B_n.
\]
Then $\vec{C}<\vec{B}$, yet arguing as in Remark \ref{rem: equi} we see that  $\vec{C}\not\prec\vec{B}$. This contradiction completes the proof of Theorem \ref{th: order}.

 \end{proof} 
 
 The partial order $\preceq$   on $\eB_n(\mu)$  is a binary relation and thus can be identified with a subset of  $\eB_n(\mu)\times \eB_n(\mu)$.     We denote  by $\preceq_t$ its (topological) closure in $\eB_n(\mu)\times \eB_n(\mu)$.  
 
 \begin{corollary} The binary relation $\preceq_t$ is a partial order relation. More precisely 
  \[
 \vec{A}\preceq_t\vec{B} \Llra \vec{A}\leq \vec{B}.
 \]
 \label{cor: order}
 \end{corollary}
 
 \begin{proof}  Clearly $\vec{A}\preceq_t\vec{B} \Longrightarrow \vec{A}\leq \vec{B}$. Conversely,  suppose $\vec{A}\leq\vec{B}$.  For every   $\ve>0$
 we  define
 \[
 \vec{B}(\ve)= \bigl(\,B_1(\ve), \dotsc , B_n(\ve)\,\bigr),
 \]
 where  $B_k(\ve)= 2^k\ve$. Then
 \[
 B_{k+1}(\ve)-B_k(\ve)= b_{k+1}+2^k\ve  > b_k+2^{k-1}\ve=B_{k}(\ve)-B_{k-1}(\ve).
 \]
  Theorem \ref{th: order} implies that $\vec{A}\prec \vec{B}(\ve)$. Letting $\ve\ra 0$  we deduce $\vec{A}\preceq_t\vec{B}$.
  
  \end{proof}

  The above  corollary can be used to produce  various interesting inequalities.
  
  For simplicity we set $\eB_n:=\eB_n(0)$. The  bead distributions in $\eB_n$ are described by      nondecreasing strings of  nonnegative numbers 
  \[
   \vec{a}=(a_1,\dotsc, a_n), \;\;0\leq a_1\leq \cdots \leq a_n
  \] To such a vector we associate the  monotone bead distribution
  \[
  \vec{A}=(A_1,\dotsc, A_n),\;\;A_k=a_1+\dotsc+a_k.
  \]
  The condition $\vec{A}\leq \vec{B}$ in $\eB_n$ can then be rewritten   as  
  \[
  a_1+\cdots+a_k\leq b_1+\cdots+b_k,\;\;\forall k=1,\dotsc, n.
  \]
 In this notation, and admissible bead slide is a transformation of the form
 \begin{equation}
 (a_1,\dotsc, a_k, a_{k+1},\dotsc, a_n)\longmapsto (a_1,\dotsc, a_k+\delta, a_{k+1}-\delta,\dotsc, a_n),\;\;2\delta \leq a_{k+1}-a_k.
 \label{eq: slide}
 \end{equation}
  Suppose  $f:[0,\infty)\ra [0,\infty)$ is a nondecreasing $C^1$ function. We then get   a map $\eT_f:\eB_n\ra \eB_n$,
  \[
 (a_1,a_1+a_2,\dotsc,a_1+\cdots + a_n)\mapsto \bigl(\, f(a_1), f(a_1)+f(a_2),\dotsc, f(a_1)+\cdots +f(a_n)\,\bigr).
  \]
 \begin{theorem} Suppose $f:[0,\infty)\ra [0,\infty)$ is  $C^1$ and nondecreasing. Then the  induced map $\eT_f:\eB_n\ra \eB_n$ preserves the order relation $\leq$ if and only if $f$ is concave, i.e., the derivative $f'$ is nonincreasing.
 \label{th: monotone}
 \end{theorem}
 
 \begin{proof} In view of  Corollary  \ref{cor: order} and   the  continuity  of $f$ we deduce that $\eT_f$  preserves the order $\leq$ if and only if  $\eT_f(\vec{A})\leq\eT_f(\vec{B})$ whenever $\vec{B}$ is obtained from $\vec{A}$ via a single admissible bead slide.  Using (\ref{eq: slide}) we see  that  this means that for any $0\leq x\leq y$, $0\leq \delta \leq \frac{1}{2}(y-x)$  we have 
 \[
 f(x+\delta)\geq f(x),\;\;f(x+\delta)+f(y-\delta)\geq f(x)+f(y).
 \]
 The  first inequality follows from the fact that $f$ is nodecreasing. The second inequality can be rephrased as
 \[
 \int_x^{x+\delta} f'(t) dt =f(x+\delta)-f(x)\geq f(y)-f(y-\delta)=\int_{y-\delta}^y f'(s) ds,
 \]
 for any $x,y,\delta\geq 0$ such that $x\leq x+\delta\leq y-\delta\leq y$.  This clearly happens if and only if $f'$ is nonincreasing.
  \end{proof}
  
  \begin{remark}  In the above result we can drop the $C^1$ assumption  on $f$, but the last  step in the proof requires a slightly   longer and less transparent argument.\qed
  \end{remark}

  \begin{corollary}  Suppose   $f:[\mu,\infty) \ra \bR$ is  $C^1$, nondecreasing and concave, and  $(y_i)_{1\leq i}$ is a nondecreasing sequence of  real numbers
  \[
  \mu\leq y_1\leq\cdots \leq y_n.
  \]
  Then for any  numbers  $x_1,\dotsc, x_n\in [\mu,\infty)$ such that
  \[
  x_1+\cdots +x_k\leq y_1+\cdots +y_k,\;\;\forall k=1,\dotsc, n
  \]
  we have
  \begin{equation}
  f(x_1)+\cdots +f(x_n)\leq f(y_1)+\cdots +f(y_n).
  \label{eq: ineq}
  \end{equation}
  \label{cor: ineq}
\end{corollary}

\begin{proof} Denote by $(x_k')$ the increasing rearrangement of the numbers $x_1,\dotsc, x_n$. Then
\[
x_1'+\cdots +x_k' \leq   x_1+\cdots +x_k\leq y_1+\cdots +y_k,\forall k=1,\dotsc, n,
\]
\[
f(x_1')+\cdots +f(x_n')=f(x_1)+\cdots +f(x_n),
\]
so  it suffices to prove (\ref{eq: ineq}) in the special case when the sequence $(x_k)$ is nondecreasing.   Define
\[
a_k:= x_k-\mu,\;\;b_k:=y_k-\mu, \;\;1\leq k\leq n,
\]
\[
A_k=a_1+\cdots +a_k,\;\;B_k=b_1+\cdots +b_k,\;\;1\leq k\leq n,
\]
\[
g:[0,\infty)\ra [0,\infty),\;\;g(t)=f(t+\mu)-f(\mu).
\]
Then  $(A_1,\dotsc, A_n)\leq (B_1,\dotsc, B_n)\in\eB_n$, and the function $g$ is $C^1$, nondecreasing and concave.   It follows that the   induced map $\eT_g:\eB_n\ra\eB_n$ is    order preserving. In particular, we conclude that
\[
g(a_1)+\cdots +g(a_n)\leq g(b_1)+\cdots +g(b_n).
\]
This clearly implies (\ref{eq: ineq}).

\end{proof}

\begin{corollary} Suppose  $f:\bR\ra \bR$ is a $C^1$, concave function and  $ y_1\leq \cdots \leq y_n$. Then for any sequence $x_1,\dotsc, x_n$ such that
\[
x_1+\cdots + x_k\leq y_1+\cdots + y_k,\;\;\forall k=1,\dotsc, n-1,
\]
and
\begin{equation}
x_1+\cdots +x_n=y_1+\cdots +y_n
\label{eq: schur}
\end{equation}
we have
\begin{equation}
f(x_1)+\cdots +f(x_n)\leq	 f(y_1)+\cdots +f(y_n).
\label{eq: schur1}
\end{equation}
\label{cor: schur}
\end{corollary}

\begin{proof} Choose $L>\max\{x_i,y_j;\;\;1\leq i,j\leq n\}$ and   define
\[
g:\bR\ra \bR,\;\;g(t)=\begin{cases} 
f(t)-f'(L)t, &  t\leq L\\
f(L)-f'(L)L, & t>L.
\end{cases}
\]
Then  $g$ is $C^1$, nondecreasing and concave and Corollary \ref{cor: ineq} implies that
\[
f(x_1)+\cdots +f(x_n)-f'(L)\sum_{k=1}^n x_k\leq f(y_1)+\cdots +f(y_n)-f'(L)\sum_{k=1}^ny_k.
\]
The inequality (\ref{eq: schur1}) now follows by invoking the equality (\ref{eq: schur}).

\end{proof}

   Corollary \ref{cor: schur}  implies  the    Schur majorization inequalities \cite[2.19-20]{HLP}, \cite[Chap. 13]{Ste}. More precisely, we have the following result.
   
   \begin{corollary}[Schur majorization] Suppose $b_1\geq \cdots \geq b_n$ is a nonincreasing  sequence of real numbers and $g:\bR\ra \bR$ is a  $C^1$, convex function, i.e. , $g'$ is nondecreasing. Then for any sequence $a_1,\dotsc, a_n$ satisfying
   \[
   a_1+\cdots +a_k\geq b_1+\cdots +b_k,\;\;k=1,\dotsc, n-1,
   \]
   and
   \[
   a_1+\cdots +a_n=b_1+\cdots +b_n
   \]
   we have
   \[
   g(a_1)+\cdots +g(a_n)\geq g(b_1)+\cdots +g(b_n).
   \]
   \end{corollary}
   
   \begin{proof} Use Corollary \ref{cor: schur} with the sequences $x_k=-a_k$, $y_j=-b_j$ and $f(t)=-g(-t)$.
   \end{proof}

\end{document}